\documentclass[12pt]{article}

\usepackage[margin=1in]{geometry}  
\usepackage{graphicx}              
\usepackage{amsmath}               
\usepackage{amsfonts}              
\usepackage{amsthm}                
\usepackage{color}
\usepackage{tabulary}
\usepackage{slashbox}

\newtheorem{thm}{Theorem}
\newtheorem{lem}[thm]{Lemma}
\newtheorem{prop}[thm]{Proposition}

\DeclareMathOperator{\PC}{PSL_2(\CC)}

\DeclareMathOperator{\HHH}{\mathcal{H}}

\DeclareMathOperator{\CP1}{\mathbb{C}\mathbb{P}^1}      %

\DeclareMathOperator{\lcm}{lcm}

\DeclareMathOperator{\RE}{Re}
\DeclareMathOperator{\IM}{Im}

\newcommand{\NN}{\mathbb{N}}      %
\newcommand{\ZZ}{\mathbb{Z}}      
\newcommand{\RR}{\mathbb{R}}      
\newcommand{\QQ}{\mathbb{Q}}      
\newcommand{\CC}{\mathbb{C}}      
\newcommand{\HH}{\mathbb{H}}      

\begin{document}


\title{Introduction to Chabauty topology and \\Pictures of the Chabauty space of $\CC^\ast$}

\author{Hyungryul Baik \& Lucien Clavier \\Department of Mathematics\\310 Malott Hall, Cornell University \\Ithaca, New York 14853-4201 USA}

\maketitle

\begin{abstract} The aim of this paper is to provide a gentle introduction to Chabauty topology, while very little background knowledge is assumed. As an example, we provide pictures for the Chabauty space of $\CC^\ast$.
Note that the description of this space is not new; however the pictures are novel.
\end{abstract}

\section{Hausdorff distance and Chabauty Space} 
Suppose you have a metric space $(X, d)$, and two closed subsets $A$ and $B$ of $X$.
 When can we say $A$ is close to $B$? 
 First note that the following value is small exactly when some point of $A$ is close to some point of $B$:
 $$d(A,B) := \inf_{a \in A, b \in B} d(a, b). $$ 
 For instance, if $A$ and $B$ intersect, the above quantity is zero.
Geometrically speaking, we would like to say that $A$ is close to $B$ if every point in $A$ is close to some point in $B$, and every point in $B$ is closed some point in $A$.
 With this in mind, we can define the distance between $A$ and $B$ to be:
 $$ d_H(A,B) = \max\{ \sup_{a \in A} \inf_{b \in B} d(a,b), \sup_{b \in B} \inf_{a \in A} d(a,b)\},$$ and we say $A$ and $B$ are close to each other if $d_H(A, B)$ is small.
 This new quantity measures exactly what we wanted:
 $\sup_{a \in A} \inf_{b \in B} d(a,b)$ is small if every point in $A$ is close to some point in $B$, and $ \sup_{b \in B} \inf_{a \in A} d(a,b)$ is small if every point in $B$ is close to some point in $A$.
 As a consequence, if $d_H(A,B)$ is small, then $A$ and $B$ look very similar as geometric objects in $X$.
 Let $F(X)$ be the set of all closed subsets of $X$ and let us measure how far any two elements of $F(X)$ are apart with $d_H$.
 This is really a distance when $X$ is compact (recall that any closed subspace of a compact space is compact, so that $F(X)$ is now the set of all compact subsets of $X$).
 In this case, we call $d_H$ \emph{the Hausdorff distance}.
 The topology of $F(X)$ induced by the metric $d_H$ is not easy to understand in general.
 For instance, when $X$ is simply a closed interval, $F(X)$ is already homeomorphic to a Hilbert cube.
 See the beautiful paper \cite{West} by R. Schori and J. West.

We can still make $d_H$ as a metric for some non-compact metric spaces $X$.
 More precisely, we only need $X$ to admit a one-point compactification which is metrizable.
 There are many examples of such spaces.
 The following lemma gives you one criterion.
 \begin{lem} \label{oneptcompact} Let $X$ be a second-countable, locally compact metric space.
 Then its one-point compactification $\overline{X}$ is metrizable.
 \end{lem} For the proof of this lemma, see Section 2 of \cite{BC1}.
 From now on, let $X$ be a non-empty metric space which admits a one-point metrizable compactification $\overline{X}$.
 Let $P_\infty$ be the point at infinity added to $X$ to form $\overline{X}$.
 Then for each element $A$ of $F(X)$, we can compactify it as a subset of $\overline{X}$ by adding $P_\infty$ to $A$.
 Let $\overline{A}$ be the resulting compact subset of $\overline{X}$.
 We compactify all elements of $F(X)$ in this way simultaneously.
 Namely, let $\overline{F(X)}$ be defined by $$ \overline{F(X)} = \{ \overline{A} = A \cup P_\infty : A \in F(X) \}.
 $$ Since $\overline{X}$ is metrizable, $d_H$ can be defined on $\overline{F(X)}$ with this metric on $\overline{X}$.
 Now $(\overline{F(X)}, d_H)$ is a metric space such that if $d_H(\overline{A}, \overline{B})$ is small, then $A$ and $B$ look similar in $X$! This simple-minded trick of adding one extra point to every set simultaneously is very useful to study subsets of non-compact metric spaces geometrically.

Instead of considering an arbitrary metric space, we can consider locally compact Lie groups.
If you do not know what Lie groups are,
 just keep in mind that these groups happen to be metric spaces with metrizable one-point compactification.
 Let $G$ be such a group.
 Since it is a group, it makes sense to consider its closed subgroups instead of closed subsets.
 Let $F(G)$ be the set of all closed subgroups of $G$.
 Then we can define the metric space $(\overline{F(G)}, d_H)$ as above.
 This is called the \emph{Chabauty space} of $G$.
 The \emph{Chabauty topology} is the topology on $\overline{F(G)}$ induced by the metric $d_H$.
 It was introduced by Claude Chabauty in 1905 \cite{Cha1} to study the space of all closed subgroups of a locally compact group.
 Interested readers are referred to the very nice introduction to Chabauty topology \cite{Harpe} by Pierre de la Harpe.

To grasp some feeling about what Chabauty spaces look like, we will start with a very simple example, the real line.

\section{Fundamental example: the Chabauty space of $\RR$} 
Consider the real line $\RR$ as an additive group.
 The closed subgroups of $\RR$ are either $\RR$ itself, or cyclic group generated by some real number.
 Let $G_r$ denote the group generated by $r$, so $G_r =r \ZZ= \{ \ldots, -2r, r, 0, r, 2r, \ldots \}$.
 Since $G_r = G_{-r}$, we may always assume that $r \ge 0$.
 Note that $G_0$ is the trivial group $\{0\}$.
 We would like to study the space of these groups in the Chabauty topology.
 As described in the previous section, we perform a simultaneous one-point compactification by adding $\infty$ to these subgroups of $\RR$.
 By Lemma \ref{oneptcompact}, $\overline{\RR} = \RR \cup \{\infty\}$ is a metric space with its induced topology.
 Let $d$ denote this metric.

Let's think about what would happen to $\overline{G_r}$ when $r$ tends to $0$.
 Algebraically, it is natural to think the limit $\lim\limits_{r \to 0} G_r$ to be $G_0$.
 On the other hand, $G_r$ becomes more dense as a subset of $\RR$ if $r$ gets smaller.
 If you observe only through a frame of fixed size, then you will see more and more points, and it is natural to expect that the geometric limit is everything.
 Indeed, we can prove the following lemma.
 \begin{lem} \label{ChabR2} $\overline{G_r}$ converges to $\overline{\RR}$ as $r \to 0$.
 \end{lem} 
It is easy prove Lemma \ref{ChabR2} rigorously and the proof will be similar to the one for Lemma \ref{ChabR}.
 But instead of repeating a similar proof twice, let's look at the Figure \ref{fig:realine1}.

\begin{figure}[ht] \begin{center} \includegraphics[scale=0.8]{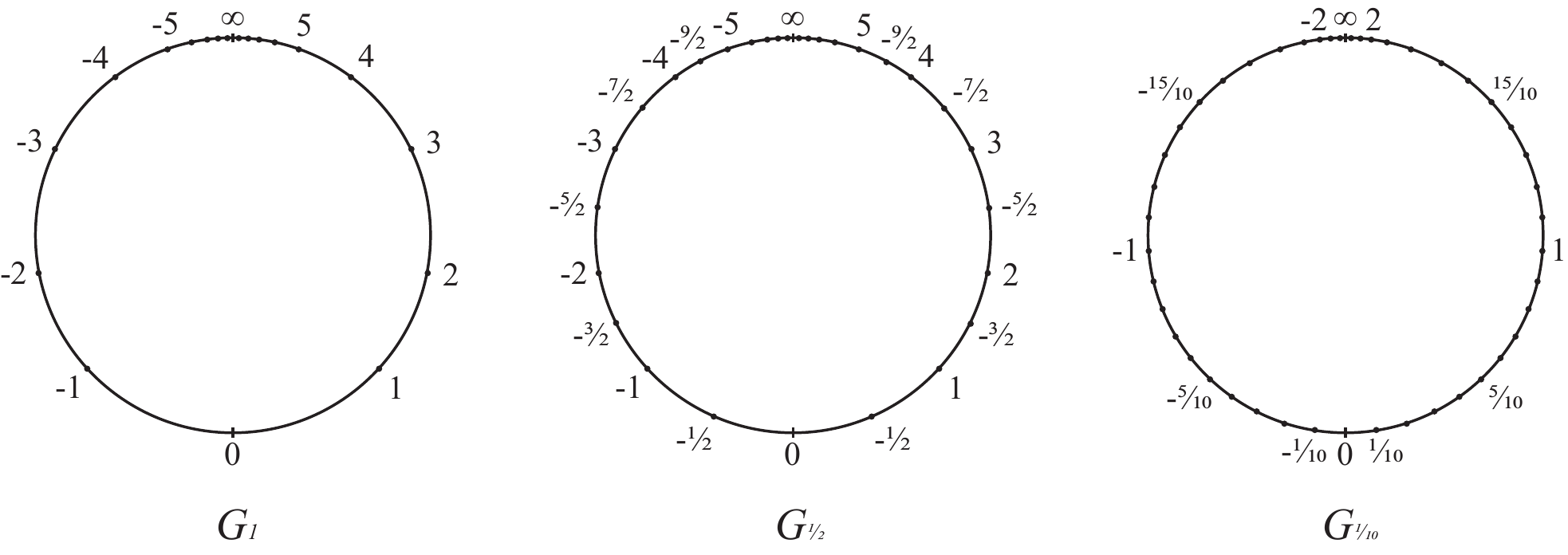} \end{center} \caption{ This shows what happens to $G_r$ when $r$ gets smaller and smaller. In $G_{1/10}$, the circle is more filled by the dots. If one looks at a fixed neighborhood of $0$, then one will see more and more points. } \label{fig:realine1} \end{figure}

We saw what would happen in Chabauty space via pictures. Now let's see how one can write this in a more precise way. The proof of the following lemma provides the way we should think about the Chabauty topology. 
In particular, note that we never need to know explicitely the metric $d$.

\begin{lem} \label{ChabR} $\overline{G_r}$ converges to $\overline{G_0} = \{0, \infty\}$ as $r \to \infty$. \end{lem} 
\begin{proof} Note that for any compact subset $K$ of $\RR$, there exists a $M >0 $ such that $G_r - \{0\} \subset K^c$ for all $r \ge M$. Also, the complements of compact subsets form a basis of neighborhood of $\infty$.

Let $\epsilon >0$ be arbitrary and $N$ be the $\epsilon$-ball around $\infty$ in $\overline{\RR}$.
 Let $M >0$ be large enough so that $G_r - \{0\} \subset N$ for all $ r \ge M$.
 For all such $r$, the Hausdorff distance between $\overline{G_r}$ and $\overline{G_0}$ is defined by \[ d_H(\overline{G_r},\overline{G_0}) = \max ( \sup_{x \in \overline{G_r} } d(x, \overline{G_0}), \sup_{y \in \overline{G_0}} d(y, \overline{G_r})). \] First, look at $ \sup_{x \in \overline{G_r} } d(x, \overline{G_0})$.
 When $x = 0$, $d(x, \overline{G_0}) = 0$, and else $d(x,\infty)\leq d(x, \overline{G_0}) \leq \epsilon$ by the choice of $r$.
 The second term $ \sup_{y \in \overline{G_0}} d(y, \overline{G_r})$ is bounded above by $\epsilon$ for the same reason, so $d_H(\overline{G_r},\overline{G_0}) \leq \epsilon$.
 Therefore, $\overline{G_r} \to \overline{G_0}$ in the Chabauty topology as $r \to \infty$.
 For a pictorial description of the proof, see Figure \ref{fig:realine2}.
 \end{proof}

\begin{figure}[ht] \begin{center} \includegraphics[scale=0.8]{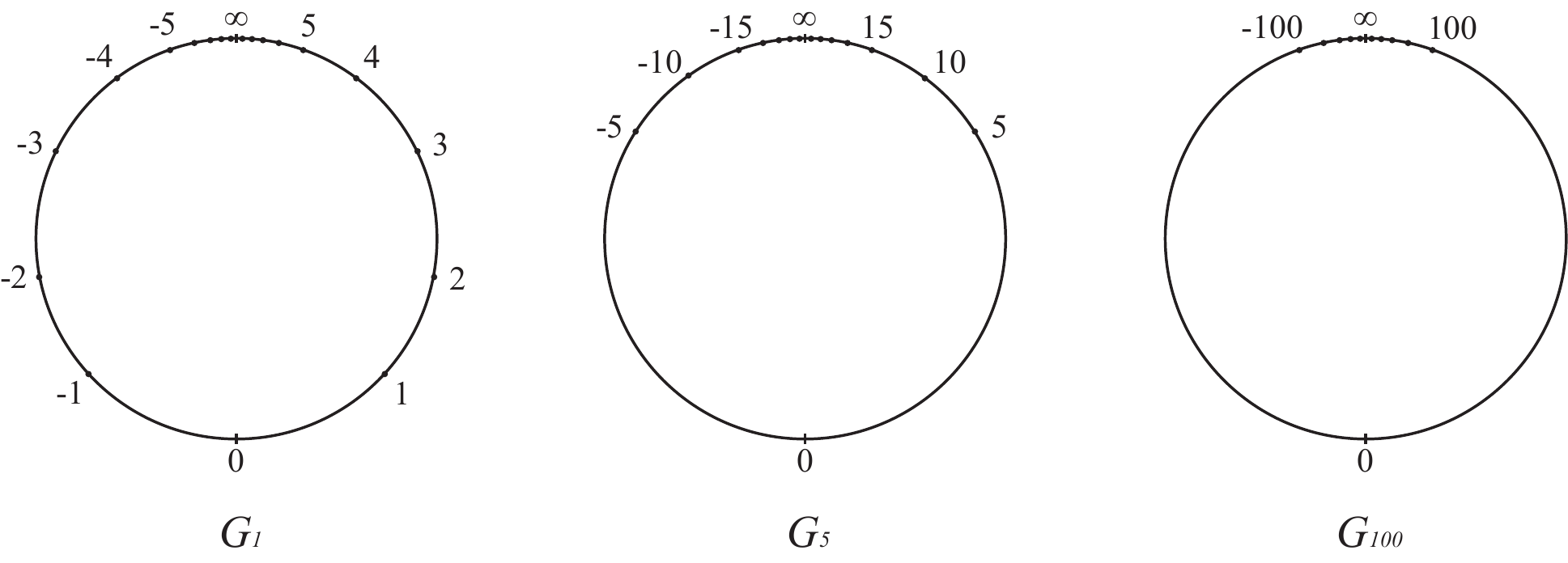} \end{center} \caption{ This shows what happens to $G_r$ when $r$ gets bigger and bigger. In $G_{100}$, most points are concentrated in a small neighborhood of $\infty$.} \label{fig:realine2} \end{figure}

Now, it is easy to deduce from the above two lemmas that the Chabauty space of $\RR$ is isomorphic to the closed interval $[0,\infty]$.
 
What about the Chabauty space of $\RR^n$ for bigger $n$?
Already the case of $\RR^2$ is quite elaborate: 
the Chabauty space of $\RR^2$ is homeomorphic to the 4-sphere
 in which sits a non-tamely embedded 2-sphere, corresponding to the set of non-lattice subgroups of $\RR^2$.
 See \cite{PourHubb} by I. Pourezza and J. Hubbard.
 A partial generalization of this result for $\RR^n$ can be found in \cite{Benoit}.

Let us now turn to the example of the Chabauty space of $\CC^\ast$.

\section{The Chabauty space of $\CC^\ast$}
 In this section, we study the Chabauty space of $\CC^\ast = \CC -\{0\}$.
 $\CC^\ast$ is an annulus with infinite modulus, so it is conformally equivalent to $\CC / 2i\pi \ZZ$ which is the model of $\CC^\ast$ we are going to use.
 We should emphasize that Lemma 1 and Propositions 2 and 3 are not original results, and can be found in \cite{HaettelRZ} or \cite{Ismael}.
 We studied this space independently along the line of work regarding \cite{BC2}, and still think the pictures we provide here are very useful to visualize the Chabauty space in this case.

Recall that any closed subgroup of $\CC$ is isomorphic to exactly one of the following groups: $\{0\}$, $\ZZ$, $\ZZ^2$, $\RR$, $\ZZ \times \RR$, $\CC$.





\begin{lem} \label{lem:subCl} 
The followings are all the closed subgroups of $\CC$ containing $2i\pi$. 
\begin{itemize} 
\item $A^{m} := (2\pi/m) i \ZZ$ for some $m \in \NN$, 
\item $B_z^{m} := z\ZZ + (2\pi/m)i\ZZ$ for some $m \in \NN$ and for $z \in \CC$ with $\RE(z) >0$ and $\IM(z) \in [0,m]$, 
\item $C_x := x \ZZ + i \RR$ for $x >0$, 
\item $D_{t}^{m} := (2\pi/m)i\ZZ + (1+i t)\RR$ with $t \in \RR$ and $m \in \NN$,
\item $A^0 = C_\infty := i \RR$, 
\item $C_0 := \CC$. 
\end{itemize} \end{lem} 

\begin{proof} Let $\Gamma$ be a closed subgroup of $\CC$ containing $2i\pi$. Then $\Gamma$ contains $A^1$. If $\Gamma$ is discrete, it must contain $ A^{m}$ for some maximal $m$; if $\Gamma$ is isomorphic to $\RR$, it must be $C_\infty=i\RR$. 

Also, if $\Gamma$ is a lattice containing $ A^{m}$ for a maximal $m$, it must be of the form $B_z^{m}$ for $z$ verifying $\RE(z) >0$ and $\IM(z) \in [0,m]$. Thus, suppose $\Gamma$ is isomorphic to $\RR \times \ZZ$. There are two cases. \newline 
Case 1: $A^{m}$ is the $\ZZ$ part. Then $\Gamma$, as a set, is the union of parallel lines of finite slope $t$ passing through the points of $A^{m}$. Thus $\Gamma$ is $D_t^{m}$. \newline 
Case 2: $A^{m}$ is contained in the $\RR$ part. Then $\Gamma$ contains $C_{\infty}$ and is the union of vertical lines equally spaced in the horizontal direction.

Hence $\Gamma = C_x$ for some $x>0$.

\end{proof}





\begin{prop} \label{prop:chabeasy} In the Chabauty topology, we have the following convergence results. \begin{itemize} \item $A^{m_n} \to \begin{cases} A^0 = i \RR \text{ if } m_n \to \infty \\A^{m} \text{ if } m_n \to m \in \NN \end{cases}$ \item $ B_{z_n}^{m_n} \to \begin{cases} C_x \text{ if } m_n \to \infty \text{ and } \RE(z_n)\to x\\A^{m} \text{ if } m_n \to m \in \NN \text{ and } \RE(z_n)\to \infty\\B_z^{m} \text{ if } m_n \to m \in \NN \text{ and } z_n\to z \text{ with }\RE(z)\in (0,\infty) \\\end{cases}$ \item $C_{x_n} \to \begin{cases} C_0 =\CC \text{ if } x_n \to 0 \\C_\infty = i \RR \text{ if } x_n \to \infty \\C_x \text{ if } x_n \to x \in (0, \infty) \end{cases}$ \item $D_{t_n}^{ m_n} \to \begin{cases} \CC \text{ if } m_n \to \infty \text{ or } t_n \to \pm \infty \\D_t^{m} \text{ if } t_n \to t \in \RR \text{ and } m_n \to m \in \NN \end{cases}$ \end{itemize} \end{prop} 

\begin{proof} The proofs of these assertions are either easier than or similar to the proof of Proposition \ref{prop:r2limits} below; they are therefore left to the reader to check. 
\end{proof}

\begin{prop} \label{prop:r2limits} Let $(z_n=x_n+ 2i\pi\theta_n)$ be a converging sequence of complex numbers, with $x_n>0$, $x_n\to 0$, $\theta_n\in [0,1]$, $\theta_n\to \theta$. If $\theta$ is rationnal, say $\theta=p/q$ with $p$, $q$ coprime positive integers, define for all $n$ $t_n$ to be the slope of the line passing through 0 and $q z_n -2i\pi p$, i.e. \[ t_n=\dfrac{2\pi}{x_n}(\theta_n-\theta ). \] Then the limit in the Chabauty topology of the sequence $ B_{z_n}^{m} $ is \[ \begin{cases} D_t^{ \lcm(m,q)} \text{ if } t_n\to t \in \RR \\\CC \text{ if } t_n\to \pm \infty \end{cases} \]

If $\theta$ is irrationnal then $ B_{z_n}^{m} \to \CC$. \end{prop} \begin{proof} The case where $\theta$ is irrationnal is immediate; let us suppose that $\theta\in \QQ$.

Note, just by drawing all lines of slope $t_n$ passing through points of $B_{z_n}^{m}$, that $B_{z_n}^{m} \subset D_{t_n}^{ \lcm(m,q)} $. Moreover, a closer look on the intersection between all those lines and the imaginary axis $i \RR$ shows that on every line of $D_{t_n}^{\frac{mq}{\gcd(pm,q)}}$ there is actually at least one point of $B_{z_n}^{m}$, as soon as $\theta_n$ is close enough to $\theta$. Since $p$ and $q$ are coprime, $\frac{mq} {\gcd(pm,q)}=\lcm(m,q)$, thus there is at least one point of $B_{z_n}^{m}$ on each line of $D_{t_n}^{\lcm(m,q)}$.

Finally, since $x_n\to 0$, we can find for every $\epsilon>0$ an integer $N$ large enough so that for all $n\geq N$, $B_{z_n}^{2 \pi/m}$ $\epsilon$-fills $D_{t_n}^{\lcm(m,q)}$ (i.e. every point of $D_{t_n}^{\lcm(m,q)}$ is at distance at most $\epsilon$ of a point of $B_{z_n}^{2 \pi/m}$). Therefore the Hausdorff distance between $B_{z_n}^{2 \pi/m}$ and $D_{t_n}^{\lcm(m,q)}$ tends to zero, and we are done. \end{proof}


Let us interpret these results geometrically (compare with \cite{HaettelRZ}).

First, let us describe the space of subgroups $ D_t^{m}$, for $m \in \NN$ and $t \in [-\infty, \infty]$. By Proposition \ref{prop:chabeasy}, $D_t^{m} \to \CC$ for any $m$ if $t \to \pm \infty$. Thus we get a bouquet of circles; one circle for each $m \in \NN$, the wedge point corresponding to the total subgroup $\CC$. We also know that when $m \to \infty$, $D_t^{m} \to \CC$ for any $t$. Thus when we increase $m$, the corresponding circle in the bouquet shrinks down to the wedge point. We call this space the $D$-bouquet. See Figure \ref{fig:Dbouquet} below.

\begin{figure}[ht] \begin{center} \includegraphics[scale=0.4]{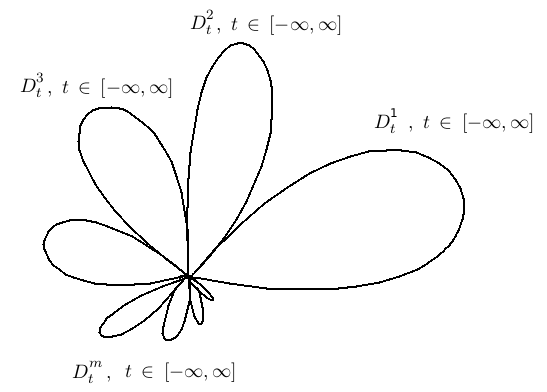} \end{center} \caption[$D$-bouquet in the Chabauty space of $\CC^*$]{The $D$-bouquet, a bouquet of circles; the wedge point represents the total subgroup~$\CC$.} \label{fig:Dbouquet} \end{figure}

Now set $m \in \NN$ to be fixed. We would like to see how the closure of the space of subgroups $B_z^{m}$ looks like, for $z$ verifying $\RE(z)>0$ and $\IM(z)\in [0,m]$. 

But since two subgroups $B_z^{m}$, $B_{z'}^{m}$ for $z,z'$ as above, are the same if and only if $z=z' \mod 2i\pi/m $, the space of $B_z^{m}$'s is the cylinder \[ \{z=x+iy ;\; x >0 ,\, y \in [0, m] \} \Big/ (x \sim x+ 2i\pi/m). \] By Proposition \ref{prop:chabeasy}, if $x \to \infty$, then $B_z^{m} \to A_m$ in the Chabauty topology. Therefore the space becomes a cone in the right direction. The identification of the other end is more complicated. Say $x \to 0$ and $y \to 2\pi \theta$ with $\theta\in [0,m]$. By Proposition \ref{prop:r2limits}, there are two cases to consider.

Case 1: $\theta$ is irrational. Then $B_z^{m}$ converges to $\CC$.

Case 2: $\theta$ is rational, say $\theta=p/q$. Then $B_z^{m} \to D_t^{\lcm(m ,q)}$ where $t=\lim \frac{2\pi} {x_n}(\theta_n-\theta )$. For $p$ and $q$ fixed, every possible limit for $t$ is possible; hence we have to blow up the point $ 0+ 2i\pi p/q$ at the left of the cylinder to a segment corresponding to $D_t^{\lcm(m ,q)}$ with $t \in [-\infty,\infty]$. Now since $D_{\pm\infty}^{\lcm(m ,q)}=\CC$, we still have to pinch the endpoints of that segment to a point, as in Figure \ref{fig:pinching} below.

\begin{figure}[ht] \begin{center} \includegraphics[scale=0.3]{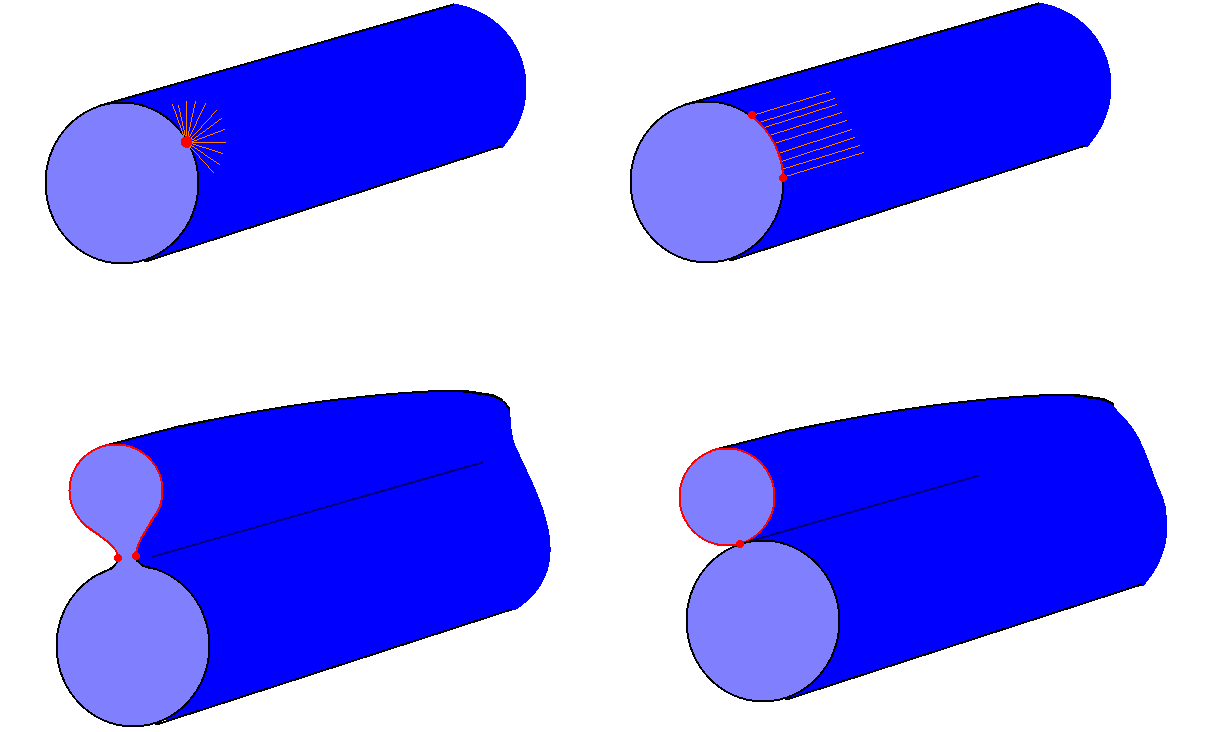} \end{center} \caption[Pinching in the $D$-bouquet]{At the left of the cylinder, each point $2i\pi\theta$ with $\theta$ rational is blown-up to a circle, resulting from a pinching. First step: $2i\pi\theta$ (in red in the first picture) is blown-up to a segment (in red in the second picture). In orange, each ray represents a locus of constant slope $t_n$ (see Proposition \ref{prop:r2limits} for notations). Second step: force the end-points of this segment to get closer and closer together (third picture) until the pinching (last picture). } \label{fig:pinching} \end{figure}

Therefore the left end of the cylinder, where $x = 0$, is glued to $D$-bouquet, in such a way that all points $2i\pi \theta$ with $\theta$ irrational are collapsed to the wedge point of the bouquet, and the other points are blown up to some circles of the bouquet.

Note that $\theta_1=p_1/q_1 $ and $\theta_2=p_2/q_2 $ are blown up to the \textit{same circle}, as long as $\lcm(m, q_1) = \lcm(m, q_2)$. Thus if $m=1$, then this end is exactly the $D$-bouquet; but whenever $m > 1$, the end is glued to a proper subbouquet, containing only petals of index in $m\NN$.

We call the resulting space the $m$th layer, noted $L_m$.

We can now collect every result of Propositions \ref{prop:chabeasy} and \ref{prop:r2limits} into a global picture for the Chabauty space of $\CC^*$ (see Figure \ref{fig:mlayer2}).

\begin{figure}[ht] \begin{center} \includegraphics[scale=0.3]{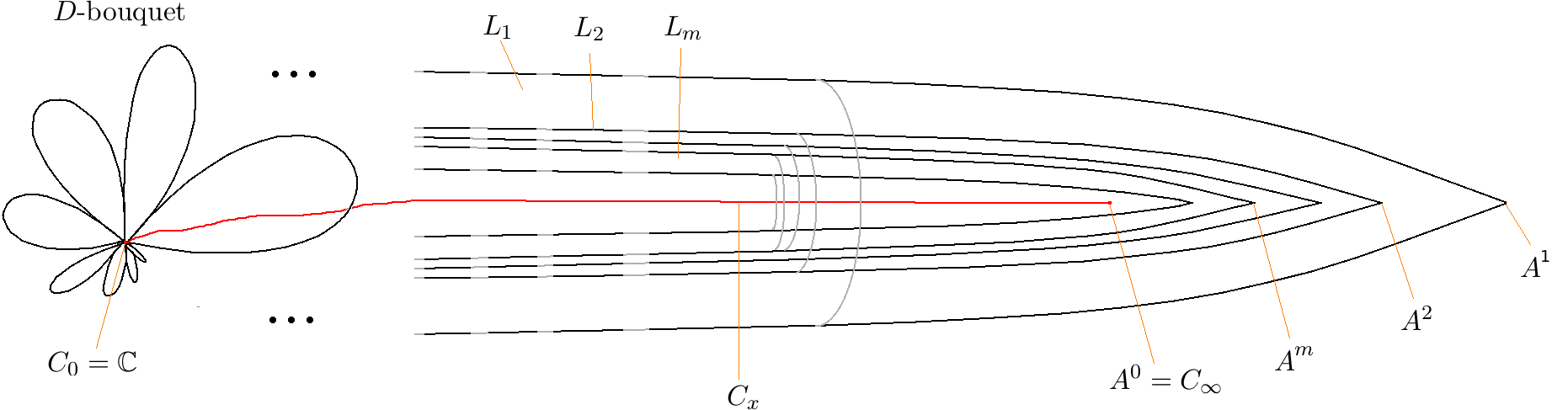} \end{center} \caption{The Chabauty space of $\CC^*$} \label{fig:mlayer2} \end{figure}

\section{Motivation for these pictures: limit of abelian subgroups of $\PC$.}

Consider the set $\HHH_l$ of all elements of $\PC$ fixing the same common axis $l$. Here, we think of $\PC$ as the group of automorphisms of the hyperbolic 3-space $\HH^3$, and $l$ is an axis in $\HH^3$, determined by its two endpoints on $\CP1=\partial \HH^3$. The set $\HHH_l$ is actually a group isomorphic to $\CC^\ast$. For example, when $l$ is the axis connecting $0$ to $\infty\in\CP1$, $\HHH_l$ is the collection of transformations $[z\mapsto\alpha z]$ for $\alpha\in\CC^\ast$.

Therefore, the set of abelian subgroups of $\PC$ sharing the same fixed axis is a copy of the Chabauty space of $\CC^\ast$, sitting inside the Chabauty space of $\PC$.

What happens to $\HHH_l$ if you continuously shrink $l$ to a point, i.e. if you make its endpoints converge to the same limit point?

It is well-known that you get some subgroups of parabolics fixing this endpoint. Since the set of parabolic elements fixing the same point is a group isomorphic to $\CC$ (for instance, the parabolic elements fixing $\infty$ are the $[z\mapsto z+\beta]$ for $\beta\in\CC$), the space of such subgroups is a copy of the Chabauty space of $\CC$, which we saw from \cite{PourHubb} to be homeomorphic to the 4-sphere.

The problem of describing how copies of the Chabauty space of $\CC^\ast$ accumulate to this 4-sphere appears of interest to us. 
It is dealt with in \cite{BC2}.

\bibliographystyle{alpha} \bibliography{biblio}

\end{document}